\documentclass[letterpaper, 10 pt, conference]{ieeeconf} 
\IEEEoverridecommandlockouts                              
 
\usepackage[latin9]{inputenc}       
\usepackage{mathtools}
\usepackage{amssymb}
\usepackage{bbold}
\usepackage[yyyymmdd,hhmmss]{datetime}
\usepackage{bm,upgreek}
\usepackage{amssymb}  
\usepackage{color}
\mathtoolsset{showonlyrefs,showmanualtags}
\usepackage{dsfont}
\usepackage{booktabs}
\usepackage[final]{graphicx}
\usepackage{nth}

\usepackage[backend=biber,style=ieee,citestyle=numeric-comp]{biblatex}

\DeclareMathOperator{\TR}{Tr}
\DeclareMathOperator{\COV}{Cov}

\newcommand{\Exp}[1]{\mathbb{E}\big[ #1\big]}

\newcommand{\B}{ \mathtt{blank} }
\usepackage[standard]{ntheorem}
\newtheorem{Assumption}{Assumption}
\newtheorem{Problem}{Problem}

\title{\LARGE \bf
MinMax Mean-Field Team Approach for  a Leader-Follower Network:  A Saddle-Point Strategy
}

\author{Mohammad M. Baharloo, Jalal Arabneydi and Amir G. Aghdam
\thanks{This work has been supported in part by the Natural Sciences and Engineering Research Council of Canada (NSERC) under Grant RGPIN-262127-17, and in part by Concordia University under Horizon Postdoctoral Fellowship.}  
\thanks{Mohammad M. Baharloo, Jalal Arabneydi, and Amir G. Aghdam are with the  Department of Electrical and Computer Engineering, 
        Concordia University, 1455 de Maisonneuve Blvd. West, Montreal, QC, Canada, Postal Code: H3G 1M8.  Email: {\tt\small baharloo@ieee.org}, {\tt\small jalal.arabneydi@mail.mcgill.ca},        
        {\tt\small aghdam@ece.concordia.ca}}%
}

\begin{document}
\maketitle

\vspace*{-5.2cm}{\footnotesize{Proceedings of IEEE IEEE Control Systems Letters, 2020.}}
\vspace*{4.45cm}

\thispagestyle{empty}
\pagestyle{empty}
\begin{abstract}
This paper investigates  a soft-constrained  MinMax control problem  of  a leader-follower network. The network consists of one leader and an arbitrary number of followers  that  wish to  reach consensus with  minimum energy consumption in the presence of external disturbances.   The  leader and  followers  are coupled   in  the dynamics and cost function. Two non-classical information structures are considered: mean-field sharing and intermittent mean-field sharing, where  the mean-field  refers to  the  aggregate state of the followers. In  mean-field sharing, every follower observes its local state, the state of the leader and the mean field while  in the intermittent mean-field sharing,  the mean-field is only  observed  at some (possibly no)  time  instants.  A social welfare cost function is defined, and it is shown that a unique saddle-point strategy exists which  minimizes the worst-case  value of the cost function  under  mean-field sharing information structure. The solution is obtained by two scalable Riccati equations, which  depend on a prescribed attenuation parameter, serving as a robustness factor. For the  intermittent mean-field sharing information structure, an approximate saddle-point strategy is proposed,   and its  converges to the  saddle-point is analyzed.  Two numerical examples are provided to demonstrate the efficacy of the  obtained results.
\end{abstract}

\section{Introduction}
Recently, there has been a surge of  interest in the application of  networked control systems in various engineering problems such as   sensor networks, swarm robotics and cooperative coordination of unmanned aerial vehicles, to name only a few.  In this type of system, it is desired to achieve a global objective (such as consensus or flocking) using local control laws with limited information exchange.  Many practical problems need to be taken into consideration in the design of a real-world networked-control system.  For instance, multi-agent networks are often  subject to external disturbances, which means   that a practical   control strategy   needs to be  robust to unwanted disturbances. However, in the theoretical analysis, such practical issues are usually neglected for simplicity.

Different robust control design techniques are introduced in the literature such as  $H_\infty$-control, risk-sensitive control  and MinMax control, each of which has its own strengths and weaknesses. For example, in risk-sensitive asset management in the financial  market,    a risk factor is utilized   in order to  capture the  randomness of the system; this   suits  applications  with  unknown disturbances. In MinMax control approach, on the other hand,   the  external disturbances are modeled as an adversarial player  attempting  to maximize the cost of the system.  In general, there are two types of formulations for the MinMax control problem:  (a)  hard-constrained formulation, where an upper bound is set on the magnitude of the disturbance,  and  (b) soft-constrained formulation that  penalizes  the disturbance by a negative quadratic cost function~\cite{van2003robust}. It is  demonstrated  in~\cite{van2003robust} that  the hard-constraint formulation is much more complex  and less tractable  than the  soft-constraint formulation. Therefore, in this paper  we   focus on  the soft-constraint formulation only.

There are two  main  challenges concerning a MinMax control  setting in the leader-follower problem. The first one is the computational complexity  that  increases with  the number of followers  (curse of dimensionality). The second challenge is   that  it is not always  feasible to assume that the states of all followers are available, specially when the number of followers is large. In such cases,  a decentralized information structure is more desirable but  it leads to a discrepancy in the followers' information. To address the above  challenges, mean-field models are introduced  in the literature   to provide a tractable approximate solution to large-scale symmetric problems~\cite{Caines2018book,gueant2011mean}.  In the control setting,  mean-field-type game~\cite{bensoussan2013mean,carmona2013control,Djehiche2017} studies an infinite-population model wherein the control optimization  problem reduces to a one-body McKean-Vlasov  optimization problem due to the simplification offered by the negligible effect in the infinite-population. In contrast, mean-field team~\cite{arabneydi2016new} focuses  on the finite-population model, where the effect of  each individual player is not necessarily negligible. For  linear quadratic mean-field teams, a transformation-based technique was first proposed in~\cite{JalalCDC2015} and further extended in~\cite{arabneydi2016new,JalalCCECE2018,JalalCDC2018}, that is similar in spirit to  the completion-of-square approach in mean-field-type game~\cite{elliott2013discrete,duncan2018linear}. In general, mean-field-type game may be viewed as a special case of mean-field team, where the information structure is mean-field sharing and the number of players is infinite.

In  the mean-field games literature, many researchers consider a leaderless network with  a large number of   homogeneous followers  in the presence of an adversarial  disturbance (player). For example,  the authors in~\cite{bauso2016opinion} study a MinMax mean-field-type game problem  in the context of social networks.  The  proposed solution is an approximate robust mean-field equilibrium that  is formulated  in terms of two coupled forward-backward partial differential equations (i.e., the Hamilton-Jacobi-Isaacs and Fokker-Planck-Kolmogorov equations).  In~\cite{huang2017robust}, a  Stackelberg-like mean-field  game problem is considered,  where an adversarial  disturbance moves first and its  worst-case value is computed irrespective of the  actions of followers. Given the  worst-case disturbance, an approximate   robust Nash equilibrium is  obtained for the followers  using two coupled forward-backward  stochastic differential equations.  A similar  Stackelberg-like mean-field   problem is investigated in~\cite{wang2017social}  with social cost function,   where  the variational analysis and  person-by-person optimality principle are employed  to derive an approximate robust mean-field-type strategy in terms of two coupled  forward-backward stochastic differential equations. The authors in~\cite{tajeddini2017robust}  integrate a leader player to the above setting in a sequential fashion, where the method is to first solve a MinMax control problem for the leader against its disturbance (irrespective of the actions of followers), and  then solve a MinMax control problem for the followers  that wish to track a convex combination  of the mean-field and the state of the leader in the presence of disturbances. Unlike the above results  in which the solution is presented in a semi-explicit manner, a leaderless MinMax mean-field-type  game solution is computed explicitly in~\cite{duncan2018linear}  in terms of two Riccati equations, where an existence condition  for the solution in the case of scalar time-invariant coefficients is proposed.

In this paper, a  leader-follower MinMax mean-field team  problem with multivariate time-varying coefficients and  an arbitrary number of homogeneous followers is investigated.  In contrast to~\cite{bauso2016opinion,huang2017robust,wang2017social,
tajeddini2017robust},  we obtain  the unique explicit saddle-point strategy  under mean-field sharing information structure, and then give an approximate solution under intermittent mean-field sharing. It is to be noted that the solution concept of a saddle-point strategy (that simultaneously solves convex and concave optimization problems), is different from a Stackelberg strategy (that solves the two problems sequentially). Also, the Riccati equations obtained in~\cite[Equation 29]{duncan2018linear} may be viewed as a special case of the continuous-time counterparts of the Riccati equations derived in this paper, where the coefficients associated with the leader are zero.

The remainder of the paper is organized as follows. The problem is defined and formulated in Section~\ref{sec:prob}.  The main results are  subsequently presented in Section~\ref{sec:main} along with the required assumptions. In Section~\ref{sec:example},  two numerical examples are given to  demonstrate the effectiveness of the results,  and finally in Section~\ref{sec:conclusion} some concluding remarks are provided.

\section{Problem Formulation} \label{sec:prob}
In this article, $\mathbb{R}$ and $\mathbb{N}$ denote, respectively, the sets of real and natural numbers. For any $k \in  \mathbb{N}$,  $\mathbb{N}_k$  denotes the finite set $\{1,\ldots,k\}$, and $x_{1:k}$ is short-hand notation for $\{x_1,\ldots,x_k\}$. $\mathbb{E}(\boldsymbol \cdot)$ is the expectation of an event, $\COV(\boldsymbol \cdot)$ is the covariance matrix of a random vector, $\TR(\boldsymbol \cdot)$ is the trace of a matrix,  and   $\mathbf{I}$ and $\mathbf{0}$ are, respectively, the identity and  zero matrices.
 
Consider a multi-agent system consisting of one leader and $n \in \mathbb{N}$ homogeneous followers.   Let $x^0_t \in \mathbb{R}^{\ell_x}$, $u^0_t \in \mathbb{R}^{\ell_u}$, $d^0_t \in \mathbb{R}^{\ell_x}$ and $w^0_t \in \mathbb{R}^{\ell_x}$ denote, respectively,  the state, action, disturbance and noise of the leader at time $t \in \mathbb{N}$, where  $\ell_x, \ell_u \in \mathbb{N}$.  Analogously,  denote by $x^i_t \in \mathbb{R}^{\ell_x}$, $u^i_t \in \mathbb{R}^{\ell_u}$, $d^i_t \in \mathbb{R}^{\ell_x}$ and $w^i_t \in \mathbb{R}^{\ell_x}$,  the state, action, disturbance and noise of follower $i \in \mathbb{N}_n$ at time $t \in \mathbb{N}$.  In addition, define the aggregate  state and aggregate action of followers as:
\begin{equation}
\bar x_t:=\frac{1}{n} \sum_{i=1}^n  x^i_t, \quad \bar u_t:=\frac{1}{n} \sum_{i=1}^n u^i_t.
\end{equation}
The dynamics of  the leader at time $t \in \mathbb{N}$ is influenced by the aggregate  state  $\bar x_t$,  the disturbance signal $d_t^0$ and noise  $w_t^0$ as follows: 
\begin{equation} \label{eq:dynamics-leader} 
x^0_{t+1}=A_t^0 x^0_t+B_t^0 u^0_t+S_t^0 \bar{x}_t + d^0_t+ w_t^0,
\end{equation}
where $A^0_t$, $B^0_t$ and  $S^0_t$ are matrices of appropriate dimensions. Furthermore, the dynamics of follower $i$ at time $t$ is affected by the state of the leader $x^0_t$, aggregate state   $\bar x_t$,  local disturbance $d^i_t$  and local noise $w_t^i$ as described  below:
\begin{equation} \label{eq:dynamics-followers} 
x^i_{t+1}=A_t x^i_t + B_t u^i_t + S_t \bar{x}_t + E_t x_t^0 + d_t^i + w_t^i,\hspace{0.3cm} i \in \mathbb{N}_n, \hspace{0.1cm} t \in \mathbb{N}, 
\end{equation}	
where $A_t$, $B_t$, $S_t$ and $E_t$ are matrices of appropriate  dimensions.  Let $T \in \mathbb{N}$ denote  the control horizon, and assume that the primitive random variables $\{ x^0_1, \{x^i_1\}_{i \in \mathbb{N}_n},  w^0_1, \{w^i_1\}_{i \in \mathbb{N}_n}, \ldots,  w^0_T, \{w^i_T\}_{i \in \mathbb{N}_n} \}$ are mutually independent. In addition, it is assumed that the local noises of followers and the noise of the leader have zero mean  and finite covariance matrices.

 To be consistent  with  the terminology  of the  mean-field teams literature~\cite{arabneydi2016new},  the  aggregate  state of  followers is called  \emph{mean-field} in this work. It is to be noted that  the term mean-field  has a slightly  different meaning in mean-field games, where it refers to the aggregate state of an infinite population (as opposed to a finite population) of followers.

In this paper, we consider two non-classical information structures: mean-field sharing (MFS) and intermittent mean-field sharing (IMFS). In  MFS information structure,  the leader has access to its local state as well as the mean-field at any time $t$, i.e., 
\begin{equation}
u^0_t=g^0_t(x^0_t,\bar x_t),
\end{equation}
where $g^0_t:\mathbb{R}^{2\ell_x}   \rightarrow \mathbb{R}^{\ell_u}$. Furthermore, each follower $i \in \mathbb{N}_n$ has access to its local state as well as the state of the leader and the mean-field at time $t$,  i.e.,
\begin{equation}
u^i_t=g^i_t(x^i_t, \bar x_t,x^0_t),
\end{equation}
where $g^i_t:\mathbb{R}^{3\ell_x} \rightarrow \mathbb{R}^{\ell_u}$. In IMFS information structure, the mean-field is observed intermittently,  i.e., 
\begin{equation}
u^0_t=g^0_t(x^0_t, z_t), \quad u^i_t=g^i_t(x^i_t,z_t,x^0_t), \quad \forall i \in \mathbb{N}_n,
\end{equation}
where $z_t:=\bar x_t$  during the time when the mean-field is observed and $z_t:=\B$  during the time when the mean-field is not observed.  In   practice, IMFS information structure is useful  when the number of  followers is neither  that small  (so that the mean-field can be shared at each time instant)  nor  is very  large  (such that the strong law of large numbers can be applied to the mean-field).   In such a case, it is  feasible to obtain the mean-field intermittently such that at some time instants the information structure  is observed  while at some others it is not. When the number of followers is large, IMFS may be reduced to no-mean-field sharing, where the mean-field is not observed at all, i.e., 
\begin{equation}
u^0_t=g^0_t(x^0_t), \quad u^i_t=g^i_t(x^i_t,x^0_t), \quad \forall i \in \mathbb{N}_n.
\end{equation}
The set  of  all control laws  $\mathbf g:=\{g^0_{1:T}, g^1_{1:T}, \ldots,g^n_{1:T}\}$ is called the strategy of the network.

\subsection{Problem statement}
 Let the set $\mathbf d$ be defined as $\{d^0_{1:T}, \{d^i_{1:T}\}_{i \in \mathbb{N}_n}\}$, and $\gamma >0$ be a given  attenuation parameter.  Then the cost function of the system  is defined as follows:
\begin{align}\label{eq:main_cost_function}
&J_n^\gamma(\mathbf g, \mathbf d) \hspace{-.1cm}= \hspace{-.1cm} \mathbb{E} (\sum_{t=1}^T\Big[\frac{1}{n} \sum_{i=1}^n\big[(x^i_t)^\intercal Q_t x^i_t+(u^i_t)^\intercal R_t u^i_t  \hspace{-.1cm}-  \hspace{-.1cm}\gamma^2 (d_t^i)^\intercal d^i_t\big]  \nonumber \\
&\quad  +(x^0_t)^\intercal Q^0_t x^0_t + (u^0_t)^\intercal R^0_t u^0_t  - \gamma^2 (d^0_t)^\intercal d^0_t \nonumber\\
&\quad + (\bar{x}_t - x^0_t)^\intercal F_t (\bar{x}_t - x^0_t)  +\bar{x}_t^\intercal P_t \bar{x}_t+\bar{u}_t^\intercal H_t \bar{u}_t\big]\Big]), 
\end{align}
where $Q_t, Q^0_t, R_t, R^0_t, F_t, P_t$ and $ H_t$ are symmetric matrices of  appropriate dimensions. 
Note that the value of $\gamma$ determines the relative importance of reaching  consensus versus rejecting  disturbance.

\begin{Problem} \label{problem1}
Find the saddle-point  strategy $\mathbf g$ under mean-field sharing information structure  such that
\begin{equation}
J^{\gamma,\ast}_n=\inf_{\mathbf g}\sup_{\mathbf d} J_{n}^\gamma(\mathbf g, \mathbf d).
\end{equation}
\end{Problem}

\begin{Problem} \label{problem2}
Find an approximate saddle-point strategy $\mathbf g_{\varepsilon}$ under intermittent mean-field sharing information structure  such that
$| \sup_{\mathbf d} J_{n}^\gamma(\mathbf g_{\varepsilon}, \mathbf d) -J^{\gamma,\ast}_n |\leq  \varepsilon(n)$,
where $\varepsilon(n) \geq 0$ and $\lim_{n \rightarrow \infty} \varepsilon(n)=0$.
\end{Problem}
\begin{remark}
Problems~1 and~2 can be extended to the case where each follower has an individual weight under the condition presented in~\cite[Assumption 3.6]{arabneydi2016new}.
\end{remark}

The main contributions of the paper are outlined below.
\begin{enumerate}
\item We present a unique saddle-point strategy  in an explicit manner for a leader-follower network under mean-field sharing information structure, where the number of followers is not necessarily large, and hence the action of a  single follower  can not be neglected  (note that   the analysis of non-negligible players is more complex  than the case with  negligible players, in general).
\item  We  propose an approximate saddle-point strategy  under intermittent  mean-field sharing information structure, whose performance is different from that of  mean-field sharing  for any finite number of followers.
\end{enumerate}

\section{Main Results}\label{sec:main}
Define the following matrices at any time $t \in \mathbb{N}_T$:
	\begin{equation}
	\bar A_t \hspace{-.1cm}:= \hspace{-.1cm}\left[ \begin{array}{cc}
	A^0_t &S^0_t\\
	E_t &A_t+S_t
	\end{array}  \right], \quad \bar B_t \hspace{-.1cm}:= \hspace{-.1cm}\left[ \begin{array}{cc}
	B^0_t &\mathbf{0}_{\ell_x \times \ell_u}\\
	\mathbf{0}_{\ell_x \times \ell_u} &B_t
	\end{array}  \right],
	\end{equation}
	\begin{equation}
	\bar Q_t \hspace{-.1cm}:= \hspace{-.1cm}\left[  \hspace{-.2cm}  \begin{array}{cc}
	Q^0_t \hspace{-.1cm}+ F_t & -F_t\\
	- F_t & Q_t  \hspace{-.1cm} +  \hspace{-.1cm}P_t+  \hspace{-.1cm}F_t
	\end{array}  \hspace{-.2cm}   \right],  \bar R_t \hspace{-.1cm}:= \hspace{-.1cm}\left[\hspace{-.2cm}  \begin{array}{cc}
	R^0_t &\mathbf{0}_{\ell_u \times \ell_u}\\
	\mathbf{0}_{\ell_u \times \ell_u} &H_t + R_t
	\end{array} \hspace{-.2cm} \right].
	\end{equation} 
	
\begin{Assumption}\label{ass:convexity}
At any time $t \in \mathbb{N}_T$, matrices  $Q_t$ and $\bar Q_t$  are positive semi-definite, and  matrices $R_t$ and $\bar R_t$ are positive definite.
\end{Assumption}

 It  will be shown later  that  Problem~\ref{problem1} under Assumption~\ref{ass:convexity}  can be cast as a strictly convex optimization problem  with respect to  the  control actions of  the leader and followers, and a strictly concave optimization problem  with respect to the disturbances.
Using the  transformation technique introduced in~\cite{arabneydi2016new},  define  the following variables: $\breve x^i_t:= x^i_t- \bar x_t$,  $\breve u^i_t:= u^i_t- \bar u_t$, $\breve d^i_t:= d^i_t- \bar d_t$ and $\breve w^i_t:= w^i_t- \bar w_t$, where 
 $\bar d_t:=\frac{1}{n} \sum_{i=1}^n d^i_t$ and $\bar w_t:=\frac{1}{n} \sum_{i=1}^n w^i_t$. It follows from~\eqref{eq:dynamics-followers} that:
	\begin{align}\label{eq:dynamics-breve}
	\bar x_{t+1} &= (A_t + S_t) \bar x_t + B_t \bar u_t + E_t x_t^0 + \bar{d}_t + \bar{w}_t, \nonumber \\
	\breve x^i_{t+1}&= A_t \breve x^i_t+ B_t \breve u^i_t + \breve{d}_t^i + \breve{w}_t^i.
	\end{align}
Note that 
	$\frac{1}{n}\sum_{i=1}^n \breve{x}^i_t=\mathbf{0}_{\ell_x \times 1}$, $\frac{1}{n}\sum_{i=1}^n \breve{u}^i_t=\mathbf{0}_{\ell_u \times 1}$, $\frac{1}{n}\sum_{i=1}^n \breve{d}^i_t=\mathbf{0}_{\ell_x \times 1}$ and $\frac{1}{n}\sum_{i=1}^n \breve{w}^i_t=\mathbf{0}_{\ell_x \times 1}$.

Define	the following variables backward in time:
	\begin{equation}\label{eq:delta-riccati}
\begin{cases}
		\breve{M}_t :=Q_{t} +A_t \breve{M}_{t+1} \breve \Delta_t^{-1} A_t^\intercal, \\
		\bar{M}_t :=\bar{Q}_{t} + \bar{A}_t \bar{M}_{t+1} \bar{\Delta}_t^{-1} \bar{A}_t^\intercal, \\
		\breve \Delta_t := \mathbf I_{\ell_x \times \ell_x} + B_t R_t^{-1} B_t^\intercal \breve{M}_{t+1} - \gamma^{-2} \breve{M}_{t+1}, \\
\bar{\Delta}_t := \mathbf I_{2 \ell_x \times 2 \ell_x} + \bar{B}_t \bar{R}_t^{-1} \bar{B}_t^\intercal \bar{M}_{t+1} - \gamma^{-2} \bar{M}_{t+1},  \\
\breve c^i_{t}:= \breve c^i_{t+1} + \TR(\breve M_{t+1} \COV{\breve w^i_t}),  \\
\bar c_{t}:=\bar c_{t+1} + \TR(\bar M_{t+1} \COV([w^0_t, \bar w_t])),
\end{cases}
\end{equation}
where $\breve{M}_{T+1}:= \mathbf{0}_{\ell_x \times \ell_x}$, $\bar{M}_{T+1}:=\mathbf{0}_{2\ell_x \times 2\ell_x}$,  $\bar c_{T+1}:=0$ and $\breve c^i_{T+1}:=0$,  $i \in \mathbb{N}_n$.	
\begin{theorem}\label{thm:1}
Let Assumption~\ref{ass:convexity} hold.   Then:
\begin{enumerate}
\item  Given any  $\gamma >0$,  Problem~\ref{problem1} admits a  unique feedback saddle-point strategy  if  matrices $ \gamma^2 \mathbf I_{\ell_x \times \ell_x} - \breve M_{t+1}$ and $\gamma^2 \mathbf I_{2\ell_x \times 2 \ell_x} - \bar M_{t+1}$ are positive definite  for every $t \in \mathbb{N}_{T-1}$, where $\breve M_{t+1}$ and $\bar M_{t+1}$ are given by~\eqref{eq:delta-riccati}.  
\item  The saddle-point strategy  of   the leader is described by:
\begin{equation}\label{eq:strategy1-optimal}
u^0_t=\bar L^{1,1}_t x^0_t + \bar L^{1,2}_t \bar x_t,
\end{equation}
and for every follower $ i \in \mathbb{N}_n$:
\begin{equation}\label{eq:strategy2-optimal}
u^i_t= \breve L_t x^i_t + \bar L_t^{2,1} x^0_t + (\bar L_{t}^{2,2} - \breve L_t) \bar x_t,
\end{equation}
where  $\left[ \begin{array}{c c}
\bar L^{1,1}_t & \bar L^{1,2}_t\\
\bar L^{2,1}_t & \bar L^{2,2}_t
\end{array} \right]:=-\bar B_t \bar M_{t+1}\bar  \Delta_t^{-1} \bar A_t$ and $\breve L_t:=-B_t \breve M_{t+1} \breve  \Delta_t^{-1} A_t$.
\item The performance  under the saddle-point strategy is given by:
$
J^{\gamma,\ast}_n= \frac{1}{n} \sum_{i=1}^n [\TR(\breve M_1 \COV(\breve x^i_1)) + \breve c^i_1] + \TR(\bar M_1 \COV([x^0_1,\bar x_1]))+ \bar c_1.$
\end{enumerate}
\end{theorem}
\begin{proof}
The proof is presented in Appendix~\ref{sec:proof}.
\end{proof}	
We now impose the following two assumptions on the model.
\begin{Assumption}\label{ass:1}
The initial states and local noises of the followers are i.i.d. random variables and independent of those of the leader. 
\end{Assumption}

\begin{Assumption}\label{ass:2}
All matrices  in the dynamic equations~\eqref{eq:dynamics-leader} and~\eqref{eq:dynamics-followers} and cost function~\eqref{eq:main_cost_function}  as well as the covariance  matrices are independent of the number of followers.
\end{Assumption}

 Let $\hat m_1$ be the expected value  of the initial states of the followers, and $\hat m_t$ denote an estimate of the mean-field $\bar x_t$ at time $t$   such that  if $z_{t+1}=\B$ at time $t+1$:
 \begin{multline}
 \hat m_{t+1}:= (A_t+ S_t + B_t \bar L^{2,2}_t) \hat m_t + (B_t \bar L^{2,1}_t + E_t) x^0_t + \bar d_t,
 \end{multline}
 and if $z_{t+1}=\bar x_{t+1}$, $ \hat m_{t+1}:=\bar x_{t+1}$.
Under Assumptions~\ref{ass:1} and~\ref{ass:2},  it can be shown that $\hat m_{t+1}$ almost surely converges to $\bar x_{t+1}$ at every time instant, as $n \rightarrow \infty$, due the strong law of large numbers, on noting that  the dynamics of the mean-field under the saddle-point  strategy is described by:
\begin{equation}
 \bar x_{t+1}= (A_t+ S_t + B_t \bar L^{2,2}_t) \bar  x_t + (B_t \bar L^{2,1}_t + E_t) x^0_t + \bar d_t + \bar w_t. 
\end{equation}
We now replace the mean-field $\bar x_t$  in the  saddle-point strategy of Theorem~\ref{thm:1}  with the estimate $\hat m_t$ (that is measurable under  IMFS information structure)\footnote{Note that the worse-case  value of $\bar d_t$ is given by~\eqref{eq:dis2} in the Appendix.} to construct the following  approximate saddle-point strategy:
\begin{equation}\label{eq:strategy1}
\begin{cases}
v^0_t=\bar L^{1,1}_t x^0_t + \bar L^{1,2}_t \hat m_t,\\
v^i_t= \breve L_t x^i_t + \bar L_t^{2,1} x^0_t + (\bar L_{t}^{2,2} - \breve L_t) \hat m_t, \quad i \in \mathbb{N}_n.
\end{cases}
\end{equation}
Since the dynamic equations~\eqref{eq:dynamics-leader} and~\eqref{eq:dynamics-followers},  cost function~\eqref{eq:main_cost_function} and  the  saddle-point  strategies~\eqref{eq:strategy1-optimal},~\eqref{eq:strategy2-optimal} and~\eqref{eq:strategy1}  are bounded and  continuous in $\bar x_t$, it results that~\eqref{eq:strategy1} is an  approximate  saddle-point strategy.  The  reader is referred to~\cite{JalalCDC2018} for  a detailed  proof   in the context of optimal control, which   is similar, to a great extent,  to the convergence proof of MinMax control problem considered   in this subsection,   but note that  the Riccati equations here are different and the  relative errors defined in~\cite{JalalCDC2018} are of  intermittent nature. However,  these differences do not add much complexity  to the convergence proof  because the Riccati equations~\eqref{eq:delta-riccati} do not depend on the number of followers  according to  Assumption~\ref{ass:2}.  Hence, the rate of convergence  with respect to the number of  followers is  $1/n$, similar to~\cite[Theorem 2]{JalalCDC2018}.  This leads to the  following theorem.

\begin{theorem}
Let Assumptions~\ref{ass:convexity}--\ref{ass:2}  hold. The  strategy~\eqref{eq:strategy1} is an approximate saddle-point strategy for Problem~\ref{problem2}.
\end{theorem}
\section{Numerical Examples}\label{sec:example}
In this section,  two numerical examples are provided to illustrate  the  efficacy of  the obtained results.

\textbf{Example 1.} Consider a multi-agent network consisting of one leader and $100$ identical followers whose dynamics are described by equations~\eqref{eq:dynamics-leader} and~\eqref{eq:dynamics-followers}, respectively,  with  the following numerical  parameters: $ A^0_t=0.85$,  $B^0_t=0.15$, $A_t=0.85$,   
$B_t=0.85$,  $S^0_t=0.03$,  $S_t=0.1$, $E_t=0.01$,  $w^i_t \sim \mathcal{N}(0,0.3)$, $\forall i \in \mathbb{N}_n$, $T=20.$

Let the initial state of the leader be $x_1^0 = 30$ and the initial states of the followers be chosen randomly (with uniform distribution) in the interval $[0,20]$.  The followers are exposed to  an external disturbance  given by:
$d_t^i = 0.6 sin(t)$,  $t \in \mathbb{N}_{20}$, $i \in \mathbb{N}_{100}$.
The objective  of the  leader  and  followers  is  to minimize the cost function~\eqref{eq:main_cost_function} under the worst-case disturbance, where at any time $t \in \mathbb{N}_{20}$: $ R_t=70$,  $Q_t= 8$,  $F_t=11$,  $P_t=0.4$,   $R_t^0=50$,   $Q^0_t=0.5$, $ H_t=0.1$.

Sample trajectories of  the leader and followers are depicted in Figure~\ref{fig1}. It is shown that as the attenuation parameter $\gamma$ increases, the fluctuations of the mean-field decrease which means better disturbance rejection.

\begin{figure}[t!]
\centering
\scalebox{.8}{
	\includegraphics[trim={3.2cm 6.3cm 0cm 6.2cm},clip,width=\linewidth]{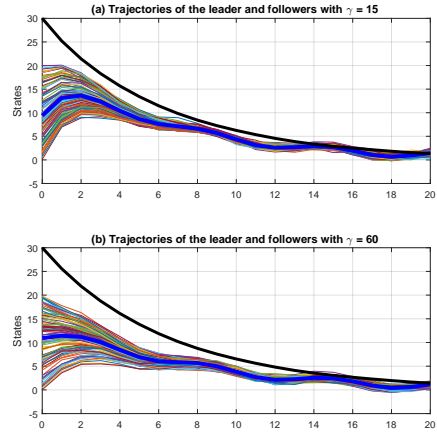} }
	\caption{ Sample trajectories of  the leader and followers in Example 1, where  thin colored  curves are the states of the  followers, thick blue curve is the mean-field,  and   thick black curve is the state of the leader.}\label{fig1}
\end{figure}

\textbf{Example 2.} Consider $100$ followers with identical dynamics that wish to track a reference signal, which may be viewed as a virtual leader with constant state  $x^0_t=10$, $\forall t \in \mathbb{N}_{30}$. The initial states of  the followers are  randomly chosen    in the interval $[0,8]$ with a  uniform distribution.  The dynamics of  the leader and followers are expressed  by the following parameters:
$A^0_t=A_t=1$,  $B^0_t=0$,  $w^i_t \sim \mathcal{N}(0,0.3)$,  $\forall i \in \mathbb{N}_{100}$, $S^0_t=0$,  $S_t=0.04$,  $
B_t=1$,  $T=30$, $ E_t=0.001.$
The followers are subject to local external disturbances given below:
$d_t^i = 0.4 sin(t)$,  $ i \in \mathbb{N}_{100}$, $t \in \mathbb{N}_{30}$.
The   weight matrices in the  cost function are given by: $ R_t=0.11$,   $Q_t= 0.01$,  $F_t=0.07$,  $P_t=0.001$,   
  $R_t^0=10^{-4}$,   $Q^0_t=10^{-4}$,  $H_t=1$.

In Figure~\ref{fig2},  three sample trajectories  of  the states of  followers are displayed for   three different  values  of the  attenuation parameter. It is  observed that as the attenuation parameter increases,  the disturbance  is rejected more strongly at the cost of prolonging the consensus process.
\begin{figure}[t!]
\centering
\scalebox{.95}{
	\includegraphics[trim={1cm 6.2cm 1cm 6.5cm},clip, width=\linewidth]{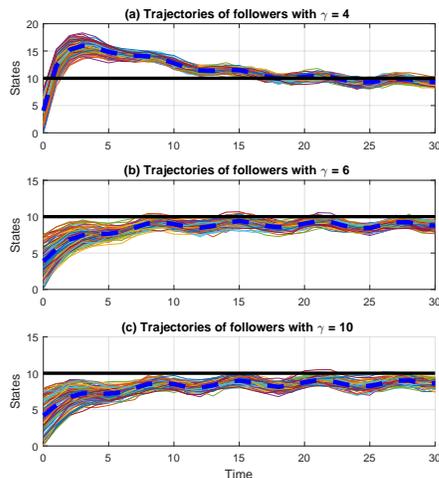} }
	\vspace{-.5cm}
	\caption{The trajectories of  followers for  different values of the attenuation parameter $\gamma$ in Example~2. }\label{fig2}
\end{figure}

\section{Conclusions}\label{sec:conclusion}
In this paper,   a robust control strategy was  proposed for a class of  leader-follower networks.  Two decentralized information structures were studied. For mean-field sharing structure, it was shown that a unique saddle-point strategy exists under  some  mild  assumptions. For  intermittent mean-field sharing, the proposed  strategy was shown to converge to the saddle-point strategy as the number of followers tends to infinity. The main feature of the obtained results is the fact that the solutions are identified by two  Riccati equations that are not in the form of forward-backward equations, and their dimensions do not increase with the number of followers. In addition, it was numerically verified that the disturbance rejection property of the solution outweighs the consensus-reaching behaviour when the attenuation parameter is large.
\printbibliography
\appendices
\section{Proof of Theorem~\ref{thm:1}}\label{sec:proof}
Rewrite the cost function  defined in~\eqref{eq:main_cost_function}  in terms of the new variables as:
	\begin{multline}
	J_n^\gamma(\mathbf g,\mathbf d) \hspace{-.1cm}=\hspace{-.1cm}\mathbb{E}(\sum_{t=1}^T\hspace{-.1cm} \Big[\frac{1}{n} \sum_{i=1}^n  (\breve x^i_t)^\intercal Q_t \breve x^i_t + (\breve u^i_t)^\intercal R_t \breve u^i_t - \gamma^2 (\breve{d}^i_t)^\intercal \breve{d}^i_t \Big]  \\
	+\left[\begin{array}{c}
	x^0_t\\
	\bar x_t
	\end{array}\right]^\intercal \hspace{-.15cm}\bar Q_t\hspace{-.1cm} \left[\begin{array}{c}
	x^0_t\\
	\bar x_t
	\end{array}\right]\hspace{-.1cm}+\hspace{-.1cm}\left[\begin{array}{c}
	u^0_t\\
	\bar u_t
	\end{array}\right]^\intercal \hspace{-.15cm}\bar R_t\hspace{-.1cm}  \left[\begin{array}{c}
	u^0_t\\
	\bar u_t
	\end{array}\right] \hspace{.0cm}-\hspace{-.0cm}\gamma^2 \hspace{-.1cm}\left[\begin{array}{c}
	d^0_t\\
	\bar d_t
	\end{array}\right]^\intercal\hspace{-.1cm}\left[\begin{array}{c}
	d^0_t\\
	\bar d_t
	\end{array}\right] ).
	\end{multline}
At any time $t$, define the following augmented vectors:
$\mathbf x_t:=[{\breve x^1_t}{}^\intercal,\ldots,{\breve x^n_t}{}^\intercal, {x^0_t}^\intercal, {\bar x_t}^\intercal]^\intercal$, 
$\mathbf u_t:=[{\breve u^1_t}{}^\intercal,\ldots,{\breve u^n_t}{}^\intercal, {u^0_t}^\intercal, {\bar u_t}^\intercal]^\intercal$, 
 $\mathbf d_t:=[{\breve d^1_t}{}^\intercal,\ldots,{\breve d^n_t}{}^\intercal, {d^0_t}^\intercal, {\bar d_t}^\intercal]^\intercal$. Suppose for now   that $\mathbf x_t$ is known, and solve the corresponding  Isaacs' equation according to~\cite[Theorem 3.2]{bacsar2008h}, i.e.,  the cost-to-go function at  the terminal time $T$ is:
$
V_{T}(\mathbf x_{T})= \frac{1}{n} \sum_{i=1}^n  (\breve x^i_{T})^\intercal \breve{M}_{T}\breve x^i_{T} + \left[\begin{array}{c}
	x^0_T\\
	\bar x_T
	\end{array}\right]^\intercal \bar M_T \left[\begin{array}{c}
	x^0_T\\
	\bar x_T
	\end{array}\right] 
	+ \frac{1}{n} \sum_{i=1}^n \breve c^i_T + \bar c_T$.  Suppose  that the cost-to-go function takes the following form at time $t+1$:
\begin{multline}\label{eq:value_function_t+1}
V_{t+1}(\mathbf x_{t+1})= \frac{1}{n} \sum_{i=1}^n  (\breve x^i_{t+1})^\intercal \breve{M}_{t+1}\breve x^i_{t+1} \\+ \left[\begin{array}{c}
	x^0_{t+1}\\
	\bar x_{t+1}
	\end{array}\right]^\intercal \bar M_{t+1} \left[\begin{array}{c}
	x^0_{t+1}\\
	\bar x_{t+1}
	\end{array}\right]+ \frac{1}{n} \sum_{i=1}^n \breve c^i_{t+1} + \bar c_{t+1}.
\end{multline}
It is now desired to show that~\eqref{eq:value_function_t+1} holds at time $t$ as well.  It follows from Isaacs' equation that:
\begin{align}\label{eq:value_function_t}
&V_{t}(\mathbf x_t) \hspace{-.1cm}= \hspace{-.1cm} \sup_{\mathbf d_t}  \inf_{\mathbf u_t} ( \frac{1}{n} \sum_{i=1}^n \big[(\breve{x}_t^i)^\intercal Q_t \breve{x}_t^i + (\breve{u}_t^i)^\intercal R_t \breve{u}_t^i - \gamma^2 (\breve{d}_t^i) ^ \intercal \breve{d}_t^i ]\nonumber \\
& +\left[\begin{array}{c}
	x^0_t\\
	\bar x_t
	\end{array}\right]^\intercal \hspace{-.15cm}\bar Q_t\hspace{-.1cm} \left[\begin{array}{c}
	x^0_t\\
	\bar x_t
	\end{array}\right]\hspace{-.1cm}+\hspace{-.1cm}\left[\begin{array}{c}
	u^0_t\\
	\bar u_t
	\end{array}\right]^\intercal \hspace{-.15cm}\bar R_t\hspace{-.1cm}  \left[\begin{array}{c}
	u^0_t\\
	\bar u_t
	\end{array}\right] \hspace{-.1cm}-\hspace{-.1cm}\gamma^2 \hspace{-.1cm}\left[\begin{array}{c}
	d^0_t\\
	\bar d_t
	\end{array}\right]^\intercal \hspace{-.15cm}\left[\begin{array}{c}
	d^0_t\\
	\bar d_t
	\end{array}\right] \nonumber \\
	&  + \Exp{V_{t+1}(\mathbf x_{t+1}) \mid \mathbf x_t, \mathbf u_t, \mathbf d_t}).
\end{align}
From~\eqref{eq:dynamics-breve},~\eqref{eq:value_function_t+1} and \eqref{eq:value_function_t}, one arrives at:
\begin{align}\label{eq:value_function_tt}
&V_{t}(\mathbf x_t)  \hspace{-.1cm}= \hspace{-.1cm} \sup_{\mathbf d_t}  \inf_{\mathbf u_t} ( \frac{1}{n} \sum_{i=1}^n \big[(\breve{x}_t^i)^\intercal Q_t \breve{x}_t^i + (\breve{u}_t^i)^\intercal R_t \breve{u}_t^i - \gamma^2 (\breve{d}_t^i) ^ \intercal \breve{d}_t^i ]\nonumber \\
& +\left[\begin{array}{c}
	x^0_t\\
	\bar x_t
	\end{array}\right]^\intercal \hspace{-.15cm}\bar Q_t\hspace{-.1cm} \left[\begin{array}{c}
	x^0_t\\
	\bar x_t
	\end{array}\right]\hspace{-.1cm}+\hspace{-.1cm}\left[\begin{array}{c}
	u^0_t\\
	\bar u_t
	\end{array}\right]^\intercal \hspace{-.15cm}\bar R_t\hspace{-.1cm}  \left[\begin{array}{c}
	u^0_t\\
	\bar u_t
	\end{array}\right] \hspace{-.1cm}-\hspace{-.1cm}\gamma^2 \hspace{-.1cm}\left[\begin{array}{c}
	d^0_t\\
	\bar d_t
	\end{array}\right]^\intercal \hspace{-.15cm}\left[\begin{array}{c}
	d^0_t\\
	\bar d_t
	\end{array}\right] \nonumber \\
	&+  \mathbb{E}[\frac{1}{n} \sum_{i=1}^n \big[(A_t \breve{x}_t^i + B_t \breve{u}_t^i + \breve{d}_t^i + \breve{w}_t^i) ^ \intercal \breve{M}_{t+1}  \nonumber 
\\
&\quad \times(A_t \breve{x}_t^i + B_t \breve{u}_t^i + \breve{d}_t^i + \breve{w}_t^i)\big] + \frac{1}{n} \sum_{i=1}^n \breve c^i_{t+1} + \bar c_{t+1}\nonumber \\
&+ (\bar A_t \left[ \begin{array}{c}
x^0_t\\
\bar x_t
\end{array}\right]+ \bar B_t \left[ \begin{array}{c}
u^0_t\\
\bar u_t
\end{array}\right]+ \left[ \begin{array}{c}
d^0_t\\
\bar d_t
\end{array}\right] + \left[ \begin{array}{c}
w^0_t\\
\bar w_t
\end{array}\right])^\intercal  \bar M_{t+1} \nonumber \\
& \quad \times (\bar A_t \left[ \begin{array}{c}
x^0_t\\
\bar x_t
\end{array}\right]+ \bar B_t \left[ \begin{array}{c}
u^0_t\\
\bar u_t
\end{array}\right]+ \left[ \begin{array}{c}
d^0_t\\
\bar d_t
\end{array}\right] + \left[ \begin{array}{c}
w^0_t\\
\bar w_t
\end{array}\right])] ).
\end{align}
This yields:
\begin{align}\label{eq:value_function_ttt}
&V_{t}(\mathbf x_t)  \hspace{-.1cm}=  \hspace{-.1cm}\sup_{\mathbf d_t}  \inf_{\mathbf u_t} ( \frac{1}{n} \sum_{i=1}^n \big[(\breve{x}_t^i)^\intercal Q_t \breve{x}_t^i + (\breve{u}_t^i)^\intercal R_t \breve{u}_t^i - \gamma^2 (\breve{d}_t^i) ^ \intercal \breve{d}_t^i] \nonumber \\
& +\left[\begin{array}{c}
	x^0_t\\
	\bar x_t
	\end{array}\right]^\intercal \hspace{-.15cm}\bar Q_t\hspace{-.1cm} \left[\begin{array}{c}
	x^0_t\\
	\bar x_t
	\end{array}\right]\hspace{-.1cm}+\hspace{-.1cm}\left[\begin{array}{c}
	u^0_t\\
	\bar u_t
	\end{array}\right]^\intercal \hspace{-.15cm}\bar R_t\hspace{-.1cm}  \left[\begin{array}{c}
	u^0_t\\
	\bar u_t
	\end{array}\right] \hspace{-.1cm}-\hspace{-.1cm}\gamma^2 \hspace{-.1cm}\left[\begin{array}{c}
	d^0_t\\
	\bar d_t
	\end{array}\right]^\intercal \hspace{-.15cm}\left[\begin{array}{c}
	d^0_t\\
	\bar d_t
	\end{array}\right] \nonumber \\
	&+ \frac{1}{n} \sum_{i=1}^n [(A_t \breve{x}_t^i + B_t \breve{u}_t^i)^\intercal  \breve M_{t+1}  (A_t \breve{x}_t^i + B_t \breve{u}_t^i)  \hspace{-.1cm}+  \hspace{-.1cm}(\breve d^i_t)^\intercal \breve M_{t+1} \breve d^i_t \nonumber \\
	&+2 (\breve d^i_t)^\intercal \breve M_{t+1} (A_t \breve x^i_t  \hspace{-.1cm}+ \hspace{-.1cm} B_t \breve u^i_t)+ 2\breve w^i_t \breve M_{t+1} (A_t \breve x^i_t + B_t \breve u^i_t + \breve d^i_t) \nonumber \\
	&+2 (\breve d^i_t)^\intercal \breve M_{t+1} \breve w^i_t  \hspace{-.1cm}+ \hspace{-.1cm} \TR(\breve M_{t+1} \COV(\breve w^i_t))]   \hspace{-.1cm}+  \hspace{-.1cm} \left[ \begin{array}{c}
 d^0_t\\
\bar d_t
\end{array}\right]^\intercal   \hspace{-.3cm} \bar M_{t+1}  \hspace{-.1cm} \left[ \begin{array}{c}
 d^0_t\\
\bar d_t
\end{array}\right] \nonumber \\
	&+  (\bar A_t \left[ \begin{array}{c}
x^0_t\\
\bar x_t
\end{array}\right]  \hspace{-.1cm}+ \hspace{-.1cm} \bar B_t \left[ \begin{array}{c}
u^0_t\\
\bar u_t
\end{array}\right])^\intercal \bar  M_{t+1}  (\bar A_t \left[ \begin{array}{c}
x^0_t\\
\bar x_t
\end{array}\right]+ \bar B_t \left[ \begin{array}{c}
u^0_t\\
\bar u_t
\end{array}\right]) \nonumber \\
&+2 \left[ \begin{array}{c}
 d^0_t\\
\bar d_t
\end{array}\right]^\intercal  \bar M_{t+1}(\bar A_t \left[ \begin{array}{c}
x^0_t\\
\bar x_t
\end{array}\right]+ \bar B_t \left[ \begin{array}{c}
u^0_t\\
\bar u_t
\end{array}\right]) + \bar c_{t+1}\nonumber \\
&+2 \left[ \begin{array}{c}
\breve w^0_t\\
\bar w_t
\end{array}\right]^\intercal  \bar M_{t+1}(\bar A_t \left[ \begin{array}{c}
x^0_t\\
\bar x_t
\end{array}\right]+ \bar B_t \left[ \begin{array}{c}
u^0_t\\
\bar u_t
\end{array}\right]) + \frac{1}{n} \sum_{i=1}^n \breve c^i_{t+1}  \nonumber \\
&+ 2 \left[ \begin{array}{c}
 d^0_t\\
\bar d_t
\end{array}\right]^\intercal  \bar M_{t+1} \left[ \begin{array}{c}
\breve w^0_t\\
\bar w_t
\end{array}\right] \hspace{-.1cm}+\hspace{-.1cm} \TR(\bar M_{t+1} \COV([w^0_t, \bar w_t]))).
\end{align}
Given any disturbance vector $\mathbf d_t$, we  now compute  the  gradient vector with respect to  $\mathbf u_t$ and  set each component to zero in order to obtain  the optimal actions of the  leader and followers, i.e., one has:
\begin{multline}
 2 (\breve{u}_t^i)^\intercal R_t + 2(\breve{u}_t^i)^\intercal B_t^\intercal \breve{M}_{t+1} B_t + 2(\breve{x}_t^i)^\intercal A_t^\intercal \breve{M}_{t+1} B_t\\ + 2(\breve{d}_t^i)^\intercal \breve{M}_{t+1} B_t = \mathbf{0}_{1\times \ell_u},\quad  \forall i \in \mathbb{N}_n,
\end{multline}
\begin{multline}
2 [{u^0_t}^\intercal, 
{\bar u_t}^\intercal] \bar{R}_t + 2[ 
{u^0_t}^\intercal,
{\bar u_t}^\intercal] \bar{B}_t^\intercal \bar{M}_{t+1} \bar{B}_t \\
+ 2 \left[ \begin{array}{c}
x^0_t\\
\bar x_t
\end{array}\right]^\intercal \bar{A}_t^\intercal \bar{M}_{t+1} \bar{B}_t + 2 \left[ \begin{array}{c}
d^0_t\\
\bar d_t
\end{array}\right]^\intercal  \bar{M}_{t+1} \bar{B}_t =  \mathbf{0}_{1\times 2\ell_u},
\end{multline}
which leads to:
\begin{multline} \label{eq:optimal_breve_u_wrt_x_d}
\breve{u}_t^{i*} = -(R_t + B_t^\intercal \breve{M}_{t+1} B_t)^{-1} (B_t^\intercal \breve{M}_{t+1} A_t)\breve{x}_t^i\\-(R_t + B_t^\intercal \breve{M}_{t+1} B_t)^{-1} (B_t^\intercal \breve{M}_{t+1})\breve{d}_t^i,
\end{multline}
and 
\begin{multline} \label{eq:optimal_u_aug_wrt_x_d}
\left[ \begin{array}{c}
u_t^{0,\ast}\\
\bar u^\ast_t
\end{array}\right]= -(\bar{R}_t + \bar{B}_t^\intercal \bar{M}_{t+1} \bar{B}_t)^{-1} (\bar{B}_t^\intercal \bar{M}_{t+1} \bar{A}_t) \left[ \begin{array}{c}
x_t^0\\
\bar x_t
\end{array}\right]\\
-(\bar{R}_t + \bar{B}_t^\intercal \bar{M}_{t+1} \bar{B}_t)^{-1} (\bar{B}_t^\intercal \bar{M}_{t+1})\left[ \begin{array}{c}
d_t^0\\
\bar d_t
\end{array}\right].
\end{multline}
It is observed that  the Hessian matrix  is diagonal, with matrices $R_t + B_t^\intercal \breve M_{t+1} B_t$ and $\bar R_t + \bar B_t^\intercal \bar M_{t+1} \bar B_t$ as its diagonal terms that are positive definite  according to Assumption~\ref{ass:convexity}.  Hence, the cost function is strictly convex in  the newly defined control actions.  

By incorporating the optimal  strategies~\eqref{eq:optimal_breve_u_wrt_x_d} and~\eqref{eq:optimal_u_aug_wrt_x_d} into~\eqref{eq:value_function_ttt},  calculating  the gradient  with respect to  $\mathbf d_t$ and setting each component to zero, one arrives at the following $n+1$ equations:
\begin{multline} 
 -\breve{M}_{t+1} B_t (R_t + B_t^\intercal \breve{M}_{t+1} B_t)^{-1} B_t^\intercal \breve{M}_{t+1} \breve{d}_t^i 
\\
-\breve{M}_{t+1} B_t (R_t + B_t^\intercal \breve{M}_{t+1} B_t)^{-1} B_t^\intercal \breve{M}_{t+1} A_t  \breve{x}_t^i \\
-\gamma ^2 \breve{d}_t^i + \breve{M}_{t+1} \breve{d}_t^i + \breve{M}_{t+1} A_t \breve{x}_t^i = \mathbf{0}_{\ell_x \times 1}, \quad \forall i \in \mathbb{N}_n,
\end{multline}
and 
\begin{multline} 
 -\bar{M}_{t+1} \bar{B}_t (\bar{R}_t +\bar{B}_t^\intercal \bar{M}_{t+1} \bar{B}_t)^{-1} \bar{B}_t^\intercal \bar{M}_{t+1}  \left[ \begin{array}{c}
d_t^0\\
\bar d_t
\end{array}\right]  \\
- \bar{M}_{t+1} \bar{B}_t (\bar{R}_t+ \bar{B}_t^\intercal \bar{M}_{t+1} \bar{B}_t)^{-1}\bar{B}_t^\intercal \bar{M}_{t+1}\bar{A}_t  \left[ \begin{array}{c}
x_t^0\\
\bar x_t
\end{array}\right] \\
-\gamma ^2  \left[ \begin{array}{c}
d_t^0\\
\bar d_t
\end{array}\right]+\bar{M}_{t+1}  \left[ \begin{array}{c}
d_t^0\\
\bar d_t
\end{array}\right] + \bar{M}_{t+1} \bar{A}_t  \left[ \begin{array}{c}
x_t^0\\
\bar x_t
\end{array}\right] =\mathbf{0}_{2 \ell_x \times 1}.
\end{multline}
After some manipulations,   the worst-case disturbances can be obtained as:
\begin{equation} \label{eq:dis1}
\breve{d}_t ^ {i*} = \gamma ^{-2} \breve{M}_{t+1} \breve \Delta_t^{-1} A_t \breve{x}_t^i,
\end{equation}
and
\begin{equation} \label{eq:dis2}
\left[ \begin{array}{c}
d_t^{0,\ast}\\
\bar d^\ast_t
\end{array}\right] = \gamma ^{-2} \bar{M}_{t+1} \bar{\Delta}_t^{-1} \bar{A}_t \left[ \begin{array}{c}
x_t^0\\
\bar x_t
\end{array}\right] ,
\end{equation}
where  $\breve \Delta_t$ and $\bar \Delta_t$ are given by~\eqref{eq:delta-riccati}.  In addition,  we  find the Hessian matrix which is diagonal with matrices $- \gamma^2 \mathbf I_{\ell_x \times \ell_x} + \breve M_{t+1}$ and $-\gamma^2 \mathbf I_{2\ell_x \times 2 \ell_x} + \bar M_{t+1}$ as its diagonal terms.   Therefore, if these matrices are negative definite, it is concluded that  the cost function is strictly concave with respect to disturbances. The recursion~\eqref{eq:delta-riccati} is finally obtained by incorporating the  worst-case disturbances~\eqref{eq:dis1} and~\eqref{eq:dis2}  into the optimal strategies~\eqref{eq:optimal_breve_u_wrt_x_d} and~\eqref{eq:optimal_u_aug_wrt_x_d} and  comparing the expressions~\eqref{eq:value_function_t+1} and~\eqref{eq:value_function_t} at times $t+1 $ and $t$, respectively, which completes the proof.

\end{document}